\documentclass[11pt,twoside]{amsart}
\usepackage{amsmath, amsthm, amscd, amsfonts, amssymb, graphicx, color}
\usepackage[bookmarksnumbered, colorlinks, plainpages]{hyperref}
\usepackage{amsmath}
\usepackage{amsfonts}

\newtheorem{theo}{Theorem}[section]
\newtheorem{pro}[theo]{Proposition}
\newtheorem{coro}[theo]{Corollary}
\newtheorem{lem}[theo]{Lemma}

\theoremstyle{definition}

\newtheorem{rem}[theo]{Remark}
\newtheorem{defi}[theo]{Definition}

\newcommand{\spec}{Spec}

\newcommand{\kdim}{\mbox{{\it k}-{\rm dim}}}

\newcommand{\be}{\begin{enumerate}}
\newcommand{\ee}{\end{enumerate}}

\author{S.M. Javdannezhad}
\address{Sayed Malek Javdannezhad, Department of Science, Shahid Rajaee Teacher Training University: Tehran, Tehran, IR}
\email{sm. javdannezhad@gmail.com}

\author{N. Shirali}
\address{Nasrin Shirali, Department of Mathematics, Shahid Chamran University of Ahvaz,  Ahvaz, Iran}
\email{shirali\_n@scu.ac.ir\\nasshirali@gmail.com}

\title[Topological expression of dual-classical Krull dimension]{A topological expression for  dual-classical Krull dimension of rings}

\begin{document}

\vspace{1.3 cm}

\thanks{{\scriptsize
\hskip -0.4 true cm MSC(2010): Primary: 16P60, 16P20 ; Secondary:
16P40.
}}

\begin{abstract}
Let $R$ be a ring and  $\mathcal X = \mathcal{SH}(R)-\{0\}$ be the set of  all   non-zero strongly hollow ideals (briefly, $sh$-ideals) of   $R$. 
We first  study the concept   $SH$-topology and investigate some of the basic properties of a topological space with this topology. It is  shown  that if  $\mathcal X $ is  with $SH$-topology, then 	$\mathcal X$ is Noetherian if and only if every subset of $\mathcal X$ is quasi-compact if and only if  $R$ has $dcc$ on semi-$sh$-ideals. 
  Finally,  the relation between the dual-classical Krull dimension of $R$ and the derived dimension of  $\mathcal X$ with a certain topology has been studied. It is proved that,  if $\mathcal X$ has derived dimension, then $R$ has the dual-classical krull dimension and in case $R$ is a $D$-ring (i.e., the lattice of ideals of $R$ is distributive), then the converse is true. Moreover these two dimension differ by at most $1$. 
 
 \end{abstract}
\maketitle
\section{Introduction and preliminaries}
The classical Krull dimension for a ring have been studied by several authors and for commutative Noetherian rings, it coincides with the Krull dimension as introduced by Gabriel and Rentschler, see\cite{go-ro}. In $1972$, Krause extended the concepts of both dimensions to arbitrary ordinals and investigated their relationship. He also, showed that a ring $R$ has the classical Krull dimension if and only if it has $acc$ on prime ideals. In $1980$, Karamzedeh \cite{kar} studied the derived dimension with respect to a certain topology, which he defined on  $X=\spec(R)$, the set of all prime ideals of $R$ and showed that $R$ has classical Krull dimension if and only if $X$ has derived dimension and these two dimensions differ by at most $1$.  
 Recently, in \cite{ja-sh}, we introduced the concept of dual-classical Krull dimension of rings for any ordinal number, on the set of strongly hollow ideals. Our definition is in the same vein, as definition of the classical Krull dimension by Krause, on the set of prime ideals, see \cite{go-ro}. We also, dualized almost all the basic results of classical Krull dimension. At the end of \cite{ja-sh},  we raised a question namely,  what is the  topological expression for the algebraic concept of dual-classical Krull dimension. In this paper, we answer this question. 
 For this purpose, we first define  $SH$-topology and investigate some of the basic properties of a topological space with this topology. Among various findings, it is  proved  that,
 $R$ has $dcc$ on $sh$-ideals if and only if	$R$ has $dcc$ on finite sum of non-comparable $sh$-ideals. 
 Also, it is shown that if  $\mathcal X $ is a topological space  with $SH$-topology,  then 	$\mathcal X$ is Noetherian if and only if every subset of $\mathcal X$ is quasi-compact if and only if  $R$ has $dcc$ on semi-$sh$-ideals. 
 Next, we define $W$-topology and the connection between  derived dimension and dual-classicall dimension is verified.
It is proved that, if $R$ is a $D$-ring and $\mathcal X = \mathcal{SH}(R)-\{0\}$ be with $W$-topology, then 
 $\mathcal X$ has derived dimension if and only if $R$ has dual-classical Krull dimension. Moreover, these two dimension differ by at most $1$. 
In  essence, we dualize almost all of the results which Karamzadeh obtained in \cite{kar}. 
It is convenient, when we are dealing with classical and dual-classical Krull dimension, to begin our list of ordinals with $1$.\\
Throughout this paper, all rings are associative with $1\neq 0$ and by an ideal we mean a two-sided one.  An ideal $S$ of a ring $R$ is small, if $S+A\neq R$, for every proper ideal $A$ of $R$. In what follows, we  recall some definitions and facts from \cite{ja-sh}, which are needed. For more details and some of  the basic facts about $sh$-ideals and the dual-classical Krull dimension of a ring, the reader is referred to \cite{ja-sh}. 

\begin{defi}
 An ideal $L$ of a ring $R$ is strongly hollow (briefly, $sh$-ideal), if $ L\subseteq A+B$ implies that $L \subseteq A$ or $L \subseteq B$, for every ideals $A,~B$ of $R$. 
\end{defi}

The notation $L \subseteq_{sh} R$  means that $L$ is an $sh$-ideal of $R$. Also,  $\mathcal{SH}(R)$ denote the set of all $sh$-ideals of $R$. Note that always, $0 \in \mathcal {SH}(R)$. Also, if $R$ is a simple ring, then 
$\mathcal{SH}(R) = \{0, R\}$. $R$ is an  $sh$-ring, if it is an $sh$-ideal or equivalently, $A+B \ne R$, for any proper ideals $A$ and $B$ of $R$. Every simple ring is an $sh$-ring.

\begin{defi}
An ideal $I$ is called semi-$sh$-ideal if it is equal to the sum of all $sh$-ideals contained in itself.
A  semi-$sh$-ring is a non-zero ring $R$, for which $R$ is an $sh$-ideal or equivalently,   $R= \sum \mathcal{SH}(R) $.
\end{defi}

 \begin{defi}\label{d3}
Let $R$ be a ring and $\mathcal Y = \mathcal Y(R)= \mathcal{SH}(R)$. Set $\mathcal Y_{-1}=
\{0\}$ and for each ordinal number $\alpha \geq 0$, let
$\mathcal Y_\alpha $ be the set of all $L \in \mathcal Y $ such
that for every $L' \in \mathcal Y$ strictly contained in $L$, there exists an ordinal $\beta < \alpha$ such that $L' \in
\mathcal Y_\beta $. Then $\mathcal Y_0  \subsetneq \mathcal Y_1
\subsetneq \mathcal Y_2  \subsetneq \dots $. The smallest ordinal
 $\alpha$, for which $\mathcal Y_\alpha = \mathcal Y$ is
called dual-classical Krull dimension of $R$ and denoted by
$d.cl. \kdim\,R$.
\end{defi}

  Note that,  $d.cl.\kdim\,R=-1$ if and only if $\mathcal Y= \{0\}$, that is $R$ has no any non-zero $sh$-ideal.  Also, $\mathcal Y_0$ consists of  $0$ and all  minimal $sh$-ideals. 
  
\begin{pro}
	Let $R$ be a ring in which  the intersection of maximal ideals  is zero. Then $d.cl.\kdim\,R \leq 0$.
\end{pro}

\begin{theo}\label{t1}
Let $R$ be a ring.
\begin{enumerate}
	\item If $R$ satisfies $dcc$ on $sh$-ideals, then $R$ has dual-classical Krull dimansion.
	\item If $R$ is a $D$-ring, the converse of $(1)$ holds.
\end{enumerate}
 \end{theo}

\section{$SH$-topology}
Let  $\mathcal X = \mathcal X(R)$ be  the set of  all   non-zero $sh$-ideals of $R$, that is $\mathcal X(R) = \mathcal{SH}(R)-\{0\}$. In this section, we introduce and study the concept of $SH$-topology.\\
We begin with the following  definition.
\begin{defi}
Let $R$ be a ring and $I$ be an ideal of $R$. We define \linebreak
$V(I) = \{L \in \mathcal X : L \subseteq I\}$ and $\underline{I} = \Sigma V(I) = \Sigma_{L \in V(I)}L$, that is $\underline{I}$ equals to the sum of all  $sh$-ideals, which contained in $I$.
\end{defi}

In what follows, we give the basic properties of these concepts.

\begin{lem}\label{l1}
	Let $R$ be a ring and  $I$, $J$ and $\{I_\gamma\}_{\gamma \in \Gamma}$ be ideals of $R$. Then 
	\be 
	\item $ V(I) = \emptyset$ if and only if $I$ does not contain any non-zero $sh$-ideal of $R$.
	\item $V(I) = \{I\}$ if and only if $I$ is a minimal $sh$-ideal of $R$.
	\item $V(I) = \mathcal X$ if and only if $I$ contains all of the $sh$-ideals of $R$.
	\item $V(I + J)= V(I) \cup V(J) $
	\item $V(\cap I_\gamma) =  \cap V(I_\gamma) $
	\item $V(I) = V(\underline I)$.
	\item $I = \underline{I}$ if and only if $I$ is a semi-$sh$-ideal.
	\item If $V(I) = V(J)$, then $\underline{I} = \underline{J}$.
	\item $\underline{\cap{I_\gamma}} = \cap \underline{I_\gamma}$.
  \ee
\end{lem}

\begin{proof}
	$(1)$ Since $\underline{I} \subseteq I$, we have $V(\underline{I}) \subseteq V(I)$. Conversely, if $L \in V(I)$, then $L$ is a $sh$-ideal contained in $I$, hence $L \subseteq \underline{I}$ and so $L \in V(\underline{I})$. Consequently, $V(I) \subseteq V(\underline{I})$ and we are done.\\
	$(2)$ It is clear.\\
	$(3)$ Clearly, $V(I_1) \cup V(I_2) \subseteq V(I_1+I_2)$. Now, let $L \in V(I_1 + I_2)$, then $L$ is a non-zero $sh$-ideal and $L \subseteq I_1+I_2$. If $L \notin V(I_1)$, then $L \nsubseteq I_1 $ and so $L \subseteq I_2$, that is, $L \in V(I_2)$. Hence, $V(I_1+I_2)  \subseteq V(I_1) \cup V(I_2)$. Thus, $V(I_1) \cup V(I_2) = V(I_1+I_2)$
\end{proof}

\begin{theo}
	Let $R$ be a ring. The sets $V(I)$, where $I$ is an ideal of $R$,  satisfy the axioms for closed sets in a toplogical space.
\end{theo}
\begin{proof}
 Clearly, $V(0) = \emptyset$, so the empty set is closed. Also, $V(R) = \mathcal X$, so the entire space is closed. Now, let $I_1, I_2$ be two ideals of $R$. By Lemma \ref{l1}(4),  $V(I_1) \cup V(I_2) = V(I_1+I_2)$, so $V(I_1) \cup V(I_2)$ is closed.  is a closed set. Finally, for any collection $\{I_\gamma\} $ of ideals, Lemma \ref{l1}(5) implies that $\cap V(I_\gamma) = V(\cap I_\gamma)$, hence $\cap V(I_\gamma)$ is closed. 
\end{proof}

 \begin{defi}
 The collection of complements of the sets $V(I)$ is a topology on $\mathcal X$ and we call it $SH$-topology.
 \end{defi}

 Recall that  a topological space $X$ is a $T_0$-space if and only if for every  dictinct points $x ,y \in X$, there is an open set containing one and not the other. Also, $X$ is a $T_1$-space if and only if for every  dictinct points $x ,y \in X$, there are open sets containing one and not the other. Finally,  $X$ is a $T_2$-space (Hausdorff space) if and only if for every  dictinct points $x ,y \in X$, there are distinct open sets $U$ and $V$  in $X$ with $x\in U$ and $y\in V$.

\begin{rem}
	The $SH$-topology on  $\mathcal X$ is trivial if and only if it is single point. For this, let   $\mathcal X$ is trivial, then, for all ideal $I$ either  $V(I)=\emptyset$ or $V(I)=\mathcal X$. Now, if $L_1 \ne L_2$ are $sh$-ideals of $R$, then $V(L_1)=\mathcal X =V(L_2)$.  So $L_1\subseteq L_2$ and $L_2\subseteq L_1$, thus $L_1=L_2$, and we are done. Conversely, if $\mathcal X=\{L\}$ then for all ideal $I$ we have $V(I)=\emptyset$ or $V(I)=\{L\}=\mathcal X$. Hence, $\mathcal X$ is a trivial space.
\end{rem}

\begin{lem}
	Let $R$ be a ring and $\mathcal X = \mathcal X(R)$  be with $SH$-topology. Then
	\be
	\item
	$\mathcal X$  is a  $T_0$-space.
	\item 
	$\mathcal X$  is a  $T_1$-space if and only if every non-zero $sh$-ideal of $R$ is a minimal $sh$-ideal if and only if every non-trivial $sh$-ideal of $R$ is a maximal $sh$-ideal. 
	\item If $\mathcal X$ is hausdorff, then there exist ideals $I_1$ and $I_2$, such that  every $sh$-ideal of $R$ is contained either in $I_1$ or $I_2$.
	
	\ee
\end{lem}

\begin{proof}
	(1) Assume that  $L_1\neq L_2$, so $L_1\nsubseteq L_2$ or $L_2\nsubseteq L_1$. If $L_1\nsubseteq L_2$, then  $L_1\in V^c(L_2)$ while $L_2\not\in V^c(L_2)$. Similarlly, if $L_2\nsubseteq L_1$, then  $L_2\in  V^c(L_1)$ while $L_1\notin V^c(L_1)$  and so $\mathcal X$ is  a $T_0$-space.\\
	(2)
	Suppose  that $\mathcal X$  is a  $T_1$-space and $L$ and $L'$ are non-zero $sh$-ideals of $R$. If $L' \subseteq L$, then  any open set  $G=V^c(I)$  that contains $L'$ will also contains $L$. Because,  $L\notin G$  implies that $L\subseteq I$ and so $L'\subseteq I$, the contradiction required.
	Conversly, let any $sh$-ideal of $R$ be a minimal $sh$-ideal, then for different $sh$-ideals of  $L_1, L_2$, we can easy to see that $L_1\in V^c(L_2)$ and  $L_2\notin V^c(L_2)$ and vice versa. This means that $\mathcal X$ is a $T_1$-space.\\
	(3) If $R$ has at most one $sh$-ideal, we are done. Now, let $L_1 \ne L_2$ be two $sh$-ideals of $R$. Since $\mathcal X$ is Hausdorff, then there exist ideals $I_1$ and $I_2$ such that, $L_1 \subseteq V^c(I_1)$ and $L_2 \subseteq V^c(I_2)$ and $V^c(I_1) \cap V^c(I_2) = \emptyset$. This implies that $ \mathcal X = V(I_1) \cup V(I_2) = V(I_1+I_2)$, hence every $sh$-ideal of $R$ is contained either in $I_1$ or $I_2$.
\end{proof}
 
\begin{pro}
	The following statements are equivalent for any ring $R$.
	\begin{enumerate}
		\item $R$ has $dcc$ on $sh$-ideals.
		\item $R$ has $dcc$ on finite sum of non-comparable $sh$-ideals. 
	\end{enumerate}
\end{pro}

\begin{proof}
	$(2)\Rightarrow (1)$ It is evident.\\
	$(1)\Rightarrow (2)$ Let $I_1 \supseteq I_2 \supseteq \dots \supseteq I_n \supseteq \dots $ be an infinite descending chain of ideals, each of which is  of the form $I_n = \varSigma_{L_i \in \Delta_n} L_i$, where $\Delta_n$ is a finite set of non-comparable $sh$-ideals. With out loss of generality, we can assume that $(i)$ $\Delta_n \cap \Delta_{n+1} = \emptyset$ (note, if $L \in \Delta_n \cap \Delta_{n+1}$, then since $\Delta_{n+1}$ is a set of non-comparable $sh$-ideals, then $L$ may not contain any other $sh$-ideal and we can remove it) and also $(ii)$ For every $L \in \Delta_n$, there exists $L' \in \Delta_{n+1}$ such that $L \supsetneq L'$ (note, otherwise $L'$ can be omitted). Hence,  we get an infinite chain $L_1 \supsetneq L_2 \supsetneq \dots \supsetneq L_n \supsetneq \dots$ of $sh$-ideals where $L_i \in \Delta_i$ for each $i$, this is a contradiction.
	\end{proof}
	
	\begin{coro}
		Let $R$ be a $D$-ring with dual classical Krull dimension. Then, $R$ has $dcc$ on finite sum of non-comparable $sh$-ideals.
	\end{coro}

	Recall that, a topological space $X$  is called Noetherian if it satisfies the ascending chain condition for open subsets.
	
\begin{pro}
	Let $R$ be a ring and $\mathcal X = \mathcal{X}(R)$ be  with $SH$-topology. The following statements are equivalent.
	\begin{enumerate}
		\item $\mathcal X$ is Noetherian.
		\item Every subset of $\mathcal X$ is quasi-compact.
		\item $R$ has $dcc$ on semi-$sh$-ideals.
	\end{enumerate} 
\end{pro}

\begin{proof}
	$(1)\Rightarrow (2)$ Let $A$ be a subset of  $ \mathcal X$ and  $\{O_\lambda\}$ be an open cover for $A$. Then $A \subseteq \cup_\lambda O_\lambda$. If $\{O_\lambda\}$ has not a finite subcover of $A$, then there exists a sequence $\lambda_1, \lambda_2, \lambda_3, \dots$ such that $O_{\lambda_1} \subsetneq O_{\lambda_1} \cup O_{\lambda_2} \subsetneq \dots \subsetneq \cup_{i=1}^n O_{\lambda_i}  \subsetneq \dots$, which is a contradiction.\\
	$(2)\Rightarrow (3)$ Let
	$I_1\supseteq I_2 \supseteq \dots \supseteq I_n \supseteq \dots$
	be an infinite descending chain of semi-$sh$- ideals, each of which is of the form $I_n = \varSigma \Delta_n$, where $\Delta_n \subseteq \mathcal X$. Then
	$V(I_1) \supseteq V(I_2) \supseteq \dots \supseteq V(I_n) \supseteq \dots$
	and so 	$\mathcal X - V(I_1) \subseteq \mathcal X - V(I_2) \subseteq \dots \subseteq \mathcal {Y}- V(I_n) \subseteq \dots$.
	Let $A = \cup (\mathcal X- V(I_n))$. Then by hypothesis, there exists $i_1, \dots i_r \in	\mathbb N$ such that
	$ A = \cup_{j=1}^r (\mathcal X- V (I_{i_i}))$ and so  $A = \mathcal {Y} - V (I_m)$,
	where $m = max\{i_1, \dots , i_r\}$. It follows that $V (I_m) = V (I_k)$ for all
	$k \geq m$. Since $I_m$ and $I_k$ are semi-$sh$-ideals, we have  
	$I_m = V (I_m) = V (I_k) = I_k$
	for all $k \geq  m$ and hence we are done.\\
	$(3)\Rightarrow (1)$ Let 
	$V(I_1) \supseteq V(I_2) \supseteq \dots \supseteq V(I_n) \supseteq \dots$ is 
	 an infinite descending chain of closed subsets of $\mathcal Y$. Hence we have
	 $\varSigma V(I_1) \supseteq \varSigma V(I_2) \supseteq \dots \supseteq \varSigma V(I_n) \supseteq \dots$.
	  and so $\underline{I_1} \supseteq \underline{I_2} \supseteq \dots \supseteq \underline{I_n} \supseteq \dots$. 
	 By $(3)$,  there exists $m \in \mathbb N$ such that $\underline{I_m} = \underline{I_k}$, thus
	 $\varSigma V(I_m) = \varSigma V(I_k)$ for all $k \geq m$. Hence $V(I_m) = V(I_k)$ for all $k \geq m$. 
	
\end{proof}

\begin{coro}
Let $R$ be a ring and $\mathcal X = \mathcal X(R)$ be with $SH$-topology. If $\mathcal X$ is quasi-compact, then $R$ has dual-classical-Krull dimension.
\end{coro}
\begin{proof}
	Since $\mathcal X$ is quasi-compact, then $R$ has $dcc$ on semi-$sh$-ideals, by the part $(3)$ of the previous proposition. Thus, it also has $dcc$ on $sh$-ideals and by Theorem \ref{t1}, it has dual-classical Krull dimension. 
\end{proof}

\section{derived dimension versus dual-classical Krull dimension}
Recall that if $A$ is a subset of a topological space $X$, then an element $x \in X$ is called a limit point for $A$ if every open set containing $A$ intersects $A$ in at least one point of $A$ distinct of $x$. The set of all limit points of $A$ is called the drived set of $A$ and is denoted by $A'$. Every $x \in A-A'$ is called an isolated point of $A$. Also, the $\alpha$-derivative of $X$ is defined by transfinite induction as follows: $X_0= X$ and $X_{\alpha+1} = X'_\alpha $ and if $\alpha$ is a limit ordinal, $X_\alpha = \bigcap _{\beta < \alpha}X_\beta$. Note that, the sets $X_\alpha$ need not to be closed in general, however, in case $X$ is hausdorff, every $X_\alpha$ is a closed. $X$ is called scattered, if $X_\alpha = \emptyset$ for some ordinal $\alpha$. If $X$ is scattered, then the smallest ordinal $\alpha$ sucah that $X_\alpha = \emptyset$ is called the derived dimension of $X$ and is denoted by $d(X) = \alpha$, see \cite{kar}. \\

Let $R$ be a ring and $\mathcal B = \{V(I): I ~\text{is an ideal of}~R\}$. For each $L \in \mathcal X$, clearly $L \in V(L)$, so $\mathcal X = \cup \mathcal B$. Also, by Lemma \ref{l1}(2), we have $V(I_1 \cap I_2) = V(I_1) \cap V(I_2)$. Hence,  $\mathcal B$ is a base for some toplology on $\mathcal X = \mathcal X(R)= \mathcal{SH}(R)-\{0\}$. The tpolology on $\mathcal X $ with $\mathcal B$ as a base, is called $W$-topology.

\begin{lem}\label{l3.1}
	Let $R$ be a ring, $\mathcal X = \mathcal X(R)$ be with $W$-topology and $S \subseteq \mathcal X$. Then an element $L \in S$ is an isolated point of $S$ if and only if it is a minimal element of $S$.
\end{lem}
\begin{proof}
	If $L \in S$ is a minimal element, then $V(L) \cap S= \{L\}$, hence $L$ is an isolated point of $S$. Conversely, let $L \in S$ be an isolated point. Then, there exists an open subset $G$ of $\mathcal X$ such that $G\cap S = \{L\}$. But there exists $V(L')$ such that $L \in V(L') \subseteq G$. This implies that $V(L') \cap S = \{L\}$. Now, let $L'' \in S$ and $L'' \subseteq L$. Then $L'' \in V(L') \cap S = \{L\}$ and so $L'' = L$. Thus, $L$ is a minimal element of $S$. 
\end{proof}


\begin{coro}\label{ipo}
Let $R$ be a ring, $\mathcal X = \mathcal X(R)$ be with $W$-topology and $S \subseteq \mathcal X$. The set of all isolated points of $S$ is open.
\end{coro}

\begin{proof}
	Suppose that $L \in S$. By the previous lemma, $L$  is a minimal element of $S$ and  then $V(L) = \{L\}\subseteq S$. According to  $W$-topology,  $V(L)$  is open and so $S$ is a open set.
\end{proof}

Let $R$ be a ring and $\mathcal X = \mathcal X(R)$. We set $\mathcal X_0 = \mathcal X$ and  for every  $\beta$, by transfinite induction, we define $\mathcal X_{\beta + 1} = \mathcal X'_\beta$,  the set of limit points of $\mathcal X_\beta$ and $\mathcal X_\beta = \cap_{\gamma < \beta} \mathcal X_\gamma$, for a limit ordinal $\beta$. Note that, $\mathcal X_0 \supseteq \mathcal X_1 \supseteq \dots \supseteq \mathcal X_n \supseteq \dots$. To see this, it sufficies to show that every $\mathcal X_\beta$ is a closed set. For this manner, we procced by transfinite induction on $\beta$. Clearly, $\mathcal X_0 = \mathcal X$ is closed. Now, let $\mathcal X_\gamma$ be closed for every  $\gamma < \beta$. If $\beta = \gamma + 1$, then $\mathcal X_\beta = \mathcal X_\gamma - \mathcal S_\gamma = \mathcal X_\gamma \cap \mathcal S_\gamma^c $, where $\mathcal S_\gamma$ is the set of all isolated points of $\mathcal X_\gamma$ which is open, by Corollary \ref{ipo}. Hence, $\mathcal S_\gamma^c$ is closed and by induction hypothesis, so is $\mathcal X_\beta$. If $\beta$ is a limit ordinal, since the intersection of any family of closed sets is closed, $\mathcal X_\beta = \cap_{\gamma < \beta} \mathcal X_\gamma$  is closed.\\

We cite the following well-known fact from \cite[Lemma 3]{kar}. 

\begin{lem}
	The following are equivalent for any toplogical space $X$.
	\begin{enumerate}
		\item Every nonempty subset of $X$ contains an isolated point.
		\item There is an ordinal $\alpha$ such that $X_\alpha = \emptyset$.
	\end{enumerate}
\end{lem}

It follows by the above lemma that if every non-empty subset of $X$, which has an isolated point, then $X$ has derived dimension. The next result is now immediate.

\begin{coro}\label{c3.3}
	Let  $R$ be a ring and $\mathcal X = \mathcal X(R)$ be with $W$-toplology. Then the following statements are equivalent.
	\begin{enumerate}
		\item $R$ has $dcc$ on $sh$-ideals.
		\item  $\mathcal X$ has derived dimension.
	\end{enumerate}
	
\end{coro}
 
 We need the following result, too.
 
\begin{theo}
	Let  $R$ be a ring and $\mathcal X = \mathcal X(R)$ be with $W$-toplology. If $\mathcal X$ has derived dimension, then $R$ has dual-classical Krull dimension. 
\end{theo}
\begin{proof}
	By Theorem \ref{t1} and Corollary \ref{c3.3}, it is evident.
\end{proof}
 
 We are now ready to prove the following proposition, which is a crucial step towads proving our main result.

\begin{pro}\label{p3.4}
  Let $R$ be a ring, $\mathcal X = \mathcal X(R)$ be with $W$-topology and $\alpha \geq 0$ be an ordinal. Then $\mathcal Y_\alpha(R) - \{0\} = \cup_{\beta \leq \alpha} S_\beta$, where $ S_\beta$ is the set of all isolated ponits of $ \mathcal X_\beta$.
\end{pro}
\begin{proof}
	We proceed by induction on $\alpha$. For $\alpha = 0$, since  $\mathcal Y_0(R) - \{0\}$ consists of all minimal $sh$-ideals of $R$, by Lemma \ref{l3.1}, we have $\mathcal Y_0(R) -\{0\} =  S_0$. Let us assume that $\mathcal Y_\gamma(R)-\{0\} = \cup_{\beta \leq \gamma} \mathcal S_\beta$ for all $\gamma < \alpha$. We show that $\mathcal Y_\alpha(R)-\{0\} = \cup_{\beta \leq \alpha} \mathcal S_\beta$. For this, let $L \in \cup_{\beta \leq \alpha} \mathcal S_\beta$. Then, $0 \ne L \in \mathcal S_\beta$, for some $\beta \leq \alpha$. If $L \in \mathcal S_\alpha$, then  by Lemma \ref{l3.1}, $L$ is a minimal element of $\mathcal Y_\alpha$. Hence, if $L' \in \mathcal X = \mathcal{SH}(R) - \{0\}$ and $L' \subsetneq L$, then $L' \notin \mathcal X_\alpha = \cap_{\beta < \alpha} \mathcal X_\beta = \mathcal X - \cup_{\beta < \alpha} \mathcal S_\beta$. This implies that  $L' \in  \mathcal S_\beta$ for some $\beta < \alpha$. Hence, $L' \in \cup_{\gamma \leq \beta}\mathcal S_\gamma$. By induction hypotesis, we have $L' \in \mathcal Y_\beta(R) - \{0\}$. This shows that $L \in \mathcal Y_\alpha(R) - \{0\}$. Now, let $L \notin \mathcal S_\alpha $. Then $L \in \mathcal S_\beta$, for some $\beta < \alpha$. This implies that $L \in \cup_{\gamma \leq \beta} \mathcal S_\gamma = \mathcal Y_\beta(R)- \{0\} \subseteq \mathcal Y_\alpha(R)-\{0\}$. Therfore, $\mathcal Y_\alpha(R)-\{0\} \supseteq \cup_{\beta \leq \alpha} \mathcal S_\beta$.\\
	Converesely, let $L \in \mathcal Y_\alpha(R) - \{0\}$. If  $L \notin \cup_{\beta < \alpha} \mathcal S_\beta$,  we show that $L \in \mathcal S_\alpha$. To this end, let $0 \ne L' \in \mathcal Y(R)$ and $L' \subsetneq L$. Then $L' \in \mathcal Y_\beta(R)-\{0\} = \cup_{\gamma \leq \beta} \mathcal S_\gamma$ for some $\beta < \alpha$. This implies that $L' \notin \mathcal X_\alpha = \mathcal X - \cup_{\gamma < \alpha} \mathcal S_\gamma$. But, $L \in \mathcal X_\alpha$ and so $L$ is a minimal $sh$-ideal of $\mathcal X_\alpha$. By Lemma \ref{l3.1}, $L \in \mathcal S_\alpha$. Therefore, $\mathcal Y_\alpha(R) - \{0\} \subseteq \bigcup_{\beta \leq \alpha} \mathcal S_\beta$ and this complete the proof.
\end{proof}

Finally, we conclude the article with the following result, which we were after.

\begin{theo}
	Let $R$ be a $D$-ring and $\mathcal X = \mathcal X(R)$ be with $W$-topology. Then 
	\begin{enumerate}
		\item $\mathcal X$ has derived dimension if and only if $R$ has dual-classical Krull dimension.
		\item $ d.cl.\kdim\,R \leq d(\mathcal X) \leq  d.cl.\kdim\,R + 1$. Moreover, if $d(\mathcal X)$ is a limit ordinal, then $d(\mathcal X) = d.cl.\kdim\,R$, otherwise, $d(\mathcal X) = d.cl.\kdim\,R + 1$
	\end{enumerate} 
\end{theo}

\begin{proof}
	
	$(1)$ By Theorem \ref{t1} and  Corollary \ref{c3.3},  is evident.\\
	$(2)$ Let $d.cl.\kdim\,R = \alpha$. Then $\mathcal Y =\mathcal Y_\alpha(R)$ and by Proposition \ref{p3.4} \linebreak  $\mathcal X = \cup_{\beta \leq \alpha} \mathcal S_\beta $. Hence, $\mathcal X_{\alpha + 1} = \mathcal X-
	\cup _{\beta \leq \alpha}S_\beta = \emptyset$. Thus $d(\mathcal X) \leq \alpha + 1$. Now, let $d(\mathcal X) < \alpha$. Then, there exists $\gamma < \alpha$ such that $\mathcal X_\gamma = \emptyset$. This implies that $\mathcal S_\gamma = \mathcal S_{\gamma + 1} = \dots = \mathcal S_\alpha = \emptyset$. Hence, $\mathcal X = \cup_{\beta \leq \alpha} \mathcal S_\beta = \cup_{\beta \leq \gamma} \mathcal S_\beta$. Consequently,  $\mathcal Y(R) = \mathcal Y_\gamma(R)$ and so  $d.cl.\kdim\,R \leq \gamma < \alpha$, which is a contradinction. Thus, $\alpha \leq d(\mathcal X)$. Therefor  $ \alpha \leq d(\mathcal X) \leq  \alpha + 1$. For the last part, we note that if $d(\mathcal X)$ is a limit ordinal, then clearly $d(\mathcal X) \ne \alpha + 1 $ and so $d(\mathcal X) = \alpha$. Finally, if $d(X) = \gamma + 1$, then $\mathcal X_{\gamma + 1} =\mathcal X-
	\cup _{\beta \leq \gamma}S_\beta= \emptyset$. Hence, $\mathcal X=
	\cup _{\beta \leq \gamma}S_\beta$. By Proposition \ref{p3.4}, $\mathcal Y_\gamma (R) =\mathcal Y$ and so $d.cl.\kdim\,R = \alpha \leq \gamma$. This implies that $\alpha + 1 \leq \gamma + 1 =d(\mathcal X)$. Thus $d(\mathcal X) = \alpha + 1$ and this complete the proof.
\end{proof}


\end{document}